\documentclass{article}

\usepackage[paper=a4paper,left=30mm,right=30mm,top=25mm,bottom=25mm]{geometry}
\usepackage{amsmath,amsfonts,amssymb,amsthm,mathtools}
\usepackage{graphicx}
\usepackage[numbers]{natbib}
\usepackage{hyperref}

\hypersetup{
    colorlinks=true,
    linkcolor=blue,
    citecolor=blue,
    urlcolor=blue
}

\numberwithin{equation}{section}

\theoremstyle{plain}
\newtheorem{theorem}{Theorem}[section]
\newtheorem{lemma}[theorem]{Lemma}
\newtheorem{proposition}[theorem]{Proposition}

\theoremstyle{definition}
\newtheorem{example}[theorem]{Example}
\newtheorem{definition}[theorem]{Definition}

\newcommand{\de}{\mathrm{\,d}}
\newcommand{\Prob}{\mathbb{P}}
\newcommand{\E}{\mathbb{E}}

\newcommand{\R}{\mathbb{R}}
\newcommand{\keywords}[1]{\par\medskip\noindent\textbf{Keywords:} #1}

\title{Kendall and Spearman bounds for Chatterjee's rank correlation under positive dependence}
\author{Marcus Rockel}
\date{\today}

\begin{document}
 
\maketitle
\begin{center}
	\small\textit{
		Department of Quantitative Finance,\\
		Institute for Economics, University of Freiburg,\\
		Rempartstr.\;16, 79098 Freiburg, Germany,\\
        \texttt{marcus.rockel@finance.uni-freiburg.de} \\[2mm]
	}
\end{center}

\begin{abstract}
We compare Chatterjee's rank correlation \(\xi\) with Kendall's \(\tau\) and Spearman's \(\rho\) under positive-dependence assumptions on bivariate copulas.
Our main technical contribution is a sharp order-violation bound for two stochastically ordered distribution functions.
This local inequality controls each conditional order-violation probability appearing in Kendall's tau by the cross-rank variance functionals that determine Chatterjee's rank correlation.
As a consequence, we prove the sharp Kendall bound
\(
    \xi(C)\leq \tau(C)
\)
for every stochastically increasing copula \(C\).
The bound is best possible: ordinal sums of product copulas attain equality.
We also prove that the weaker left-tail decreasing (LTD) and right-tail increasing (RTI) conditions jointly imply the Spearman bound
\(
    \xi(C)\leq \rho(C),
\)
with equality if and only if \(C\) is either the independence or comonotonicity copula.
Finally, checkerboard examples show that LTD or RTI alone does not imply \(\xi(C)\leq\rho(C)\), that LTD and RTI together do not imply \(\xi(C)\leq\tau(C)\), and that both bounds are directional for \(\xi\).
\end{abstract}


\section{Introduction}

Rank-based measures of dependence are central tools in nonparametric statistics.
Kendall's tau and Spearman's rho are classical rank-based measures of concordance in the sense of \cite{scarsini1984measures}: they quantify the tendency of two variables to be ordered in the same direction, and hence capture positive or negative association in terms of joint ranks.
Chatterjee's rank correlation, introduced in \cite{chatterjee2020}, is of a different nature, as it is asymmetric and measures the degree to which one variable is a measurable function of another.
Like Kendall's tau and Spearman's rho, in the case of continuous marginals, it depends only on the underlying dependence structure, and in this case it coincides with the measure in \cite{dette2013copula}.
The precise definitions are given in Section~\ref{sec:preliminaries}.

Let \(C\) be a bivariate copula, that is, a distribution function on \([0,1]^2\) with uniform marginals.
Let \(\xi(C)\) denote Chatterjee's rank correlation, \(\tau(C)\) Kendall's tau, and \(\rho(C)\) Spearman's rho.
For stochastically increasing copulas, as defined in \eqref{eq:si_definition} below, it is known that
\begin{equation}\label{eq:known-inequalities}
\tau(C)\leq\rho(C) \quad \text{and} \quad \xi(C)\leq\rho(C).
\end{equation}
The first inequality is a classical result from \cite{caperaa1993spearmans} (see also \cite[Thm.~5.2.8]{nelsen2006introduction}), while the second inequality is proved in \cite{ansari2026exact}.
Stochastic increasingness is a natural positive-dependence condition that is satisfied by many standard copula families, see \cite{ansari2023dependence} and the references therein.
Numerical evidence across a wide variety of copula families led \cite{ansari2023dependence} to conjecture that
\(
    \xi(C)\leq\tau(C)
\)
also holds for stochastically increasing copulas.
Our main result confirms this conjecture:

\begin{theorem}[SI implies \(\xi\leq\tau\)]
\label{thm:main}
If \(C\) is a stochastically increasing bivariate copula,
then
\[ 
    \xi(C)\leq \tau(C).
\]
Similarly, if \(C\) is stochastically decreasing, then \(\xi(C)\leq -\tau(C)\), and both inequalities are sharp.
\end{theorem}

The proof is based on a sharp two-sample stochastic-order inequality which is
not specific to copulas and may be of independent interest.
For two stochastically ordered laws \(F\geq G\), \(X\sim F\) and \(Z\sim G\), Theorem~\ref{thm:pairwise} bounds the pairwise order-violation probability \(\Prob(X>Z)\) by a symmetrized cross-rank variance functional of \(F\) and \(G\).
When \(F\) and \(G\) are two conditional distribution functions of an SI copula, this pairwise inequality controls the local discordance term \(D(t,s)\) in Kendall's tau from \eqref{eq:local_discordance} below.
Integrating over \(0<t<s<1\) then yields the global bound \(\xi(C)\leq\tau(C)\).
Thus, the argument proves a stronger pointwise order-violation estimate before passing to the rank-correlation comparison.
Together with \eqref{eq:known-inequalities}, we obtain the chain $0\leq \xi(C)\leq \tau(C)\leq \rho(C)$ for every stochastically increasing copula $C$.
Theorem~\ref{thm:main} strengthens the known Spearman bound \(\xi(C)\leq\rho(C)\) from \cite{ansari2026exact} to the sharper Kendall bound \(\xi(C)\leq\tau(C)\) under the same SI assumption.

We give a second strengthening of the known Spearman bound from \cite{ansari2026exact} by the following result.
It shows that the same weaker tail monotonicity assumptions of left-tail decreasingness and right-tail increasingness, as defined below in \eqref{eq:ltd_definition} and \eqref{eq:rti_definition}, already suffice for \(\xi(C)\leq\rho(C)\), and that equality is achieved only for the independence and comonotonicity copulas $\Pi$ and $M$ defined in Section~\ref{sec:preliminaries}.
Similarly, under reversed monotonicity assumptions, only the independence and countermonotonicity copulas $\Pi$ and $W$ achieve equality.
Up to interchanging the coordinates, left-tail decreasingness and right-tail increasingness are the tail conditions used by Capéraà and Genest in \cite{caperaa1993spearmans} to prove \(\tau(C)\leq\rho(C)\).
In their symmetric setting the coordinate order is immaterial.
For \(\xi\), however, the order has to agree with the direction of $\xi$.
We therefore use definitions of SI, LTD, and RTI as in \cite[Ch.~5]{nelsen2006introduction}, see Section~\ref{sec:preliminaries} below.
Counterexamples for the reversed order are given in Example~\ref{rem:si-kendall-bound-directional} and Example~\ref{rem:ltd-rti-spearman-bound-directional}.

\begin{theorem}[LTD and RTI together imply \(\xi\leq\rho\)]
\label{thm:ltd-rti-xi-rho}
Let \(C\) be left-tail decreasing and right-tail increasing.
Then
\[
    \xi(C)\leq \rho(C),
\]
with equality if and only if \(C\in\{\Pi,M\}\).
Similarly, if \(C\) is left-tail increasing and right-tail decreasing, then
\(
    \xi(C)\leq -\rho(C),
\)
with equality if and only if \(C\in\{\Pi,W\}\).
\end{theorem}

The rest of the paper is organized as follows.
Section~\ref{sec:preliminaries} contains the copula-theoretic notation, Markov-kernel representations, and the definitions of the rank coefficients and positive-dependence conditions used throughout the paper.
Section~\ref{sec:pairwise} proves the sharp order-violation bound, a standalone inequality for two stochastically ordered conditional distributions.
Section~\ref{sec:proof} applies this bound to Markov kernels of SI copulas, proves Theorem~\ref{thm:main}, and gives sharpness and directional counterexamples, including an example showing that LTD and RTI together are still insufficient for the Kendall bound.
Finally, Section~\ref{sec:ltd-rti-rho} proves the Spearman bound under the weaker LTD and RTI assumptions and shows, by examples, that neither LTD nor RTI alone suffices for this criterion.

\section{Preliminaries}
\label{sec:preliminaries}

Let \(\lambda\) denote the Lebesgue measure on \([0,1]\).
A bivariate copula is a function \(C\colon[0,1]^2\to[0,1]\) that is grounded, \(2\)-increasing, and has uniform marginals.
More explicitly, \(C(u_1,u_2)=0\) whenever \(u_1=0\) or \(u_2=0\),
\[
    C(v_1,v_2)-C(u_1,v_2)-C(v_1,u_2)+C(u_1,u_2)\geq 0
\]
for all \(u_1\leq v_1\) and \(u_2\leq v_2\), and \(C(u_1,u_2)=u_i\) whenever \(u_j=1\) for \(j\neq i\).
Such a bivariate copula \(C\) will be identified with the probability measure
\(\mu_C\) on \([0,1]^2\) whose distribution function is \(C\).
Classical copulas include $\Pi(u,v):=uv$, $M(u,v):=\min\{u,v\}$, and $W(u,v):=\max\{u+v-1,0\}$ for \((u,v)\in[0,1]^2\), which are the independence, comonotonicity, and countermonotonicity copulas, respectively.
We write \(C^\top(u,v):=C(v,u)\) for the transpose copula.

For a set $A\subseteq\R^d$, $d\in\mathbb{N}$, let $\mathcal B(A)$ denote the Borel \(\sigma\)-algebra on \(A\).
Then, let \(K_C\) be a \emph{Markov kernel} of \(C\), by which we mean a map
\(
    K_C\colon [0,1]\times\mathcal B([0,1])\to[0,1]
\)
such that \(K_C(t,\cdot)\) is a probability measure for every \(t\in[0,1]\), \(t\mapsto K_C(t,A)\) is measurable for every \(A\in\mathcal B([0,1])\), and
\begin{equation}
\label{eq:kernel-disintegration}
    \mu_C(G)
    =
    \int_0^1 K_C(t,G_t)\de t,
    \qquad
    G\in\mathcal B([0,1]^2),
\end{equation}
where
\(
    G_t:=\{y\in[0,1]:(t,y)\in G\}.
\)
Such kernels exist and positive dependence concepts in their terms are studied in \cite{fuchs2023total,fuchs2024novel}.
Equivalently, \(K_C\) is a version of the regular conditional distribution of the second coordinate given the first coordinate, see \cite[Thm.~8.5]{kallenberg2021foundations}.
Since such kernels are unique only up to \(\lambda\)-null sets in the conditioning variable, all pointwise statements involving \(K_C(t,\cdot)\) refer to the chosen version.
 
If \(C\) and \(D\) are bivariate copulas with Markov kernels \(K_C\) and
\(K_D\), respectively, define
\[
    K_{C,D}(t,A)
    :=
    \int_0^1 K_D(x,A)\,K_C(t,\de x),
    \qquad
    A\in\mathcal B([0,1]).
\]
By the standard composition rule for probability kernels, \(K_{C,D}\) is again a probability kernel, see, e.g., \cite[Lem.~3.3]{kallenberg2021foundations}. The \emph{Markov product}
\(C*D\) is then given by
\begin{equation}
\label{eq:markov-product-kernel}
    (C*D)(u,v)
    =
    \int_0^u K_{C,D}(t,[0,v])\de t
    =
    \int_0^u\int_0^1 K_D(x,[0,v])\,K_C(t,\de x)\de t,
    \quad
    (u,v)\in[0,1]^2.
\end{equation}
The resulting copula is independent of the chosen kernel versions, \(K_{C,D}\) is a Markov kernel of \(C*D\), and \eqref{eq:markov-product-kernel} is the Markov product in the sense of \cite{darsow1992copulas}, see also \cite{durante2016principles,siburg2021stochastic}.

Next, introduce the following notation for the conditional laws of \(C\).
For \(t\in[0,1]\), write
\begin{equation}\label{eq:h-definitions}
    H_t:=K_C(t,\cdot),
    \qquad
    h_v(t):=K_C(t,[0,v]),
    \qquad 0\leq v\leq1.
\end{equation}
The map
\(
    (t,v)\mapsto h_v(t)
\)
is Borel measurable.
To see joint measurability, note that for \(v<1\),
\(
    h_v(t)
    =
    \inf_{q\in\mathbb Q\cap[0,1]}
    \left(h_q(t)+\mathbf 1_{\{q\le v\}}\right).
\)
Hence \((t,v)\mapsto h_v(t)\) is a countable infimum of Borel functions for \(v<1\), and \(h_1(t)=1\).
For fixed \(v\in[0,1]\), applying \eqref{eq:kernel-disintegration} to \(G=[0,u]\times[0,v]\) gives
\[
    C(u,v)
    =
    \int_0^u K_C(t,[0,v])\,\de t,
    \qquad 0\leq u\leq1.
\]
Hence \(u\mapsto C(u,v)\) is absolutely continuous, and the Newton--Leibniz formula for the Lebesgue integral yields
\begin{equation}\label{eq:first_derivative_identity}
    h_v(t)=K_C(t,[0,v])=\partial_1 C(t,v)
    \qquad\text{for \(\lambda\)-a.e.~}t,
\end{equation}
see \cite[Vol.~I, Thms.~5.3.6 and~5.4.2]{bogachev2007measure}.
For each fixed \(v\), the identity \(h_v=\partial_1C(\cdot,v)\) holds \(\lambda\)-a.e. Since both sides are jointly measurable in \((t,v)\), Fubini's theorem yields the identity for \(\lambda^2\)-a.e. \((t,v)\).
Since the second marginal of \(C\) is uniform, the disintegration formula from \eqref{eq:kernel-disintegration} implies that the mixture of the conditional laws \(H_t\) is \(\lambda\):
for every \(A\in\mathcal B([0,1])\),
\[
    \int_0^1 H_t(A)\de t
    =
    \mu_C([0,1]\times A)
    =
    \lambda(A).
\]
Equivalently, for every nonnegative Borel function
\(\varphi:[0,1]\to[0,\infty]\),
\begin{equation}
\label{eq:mixture-uniform-functions}
    \int_0^1\int_0^1 \varphi(y)\,H_t(\de y)\de t
    =
    \int_0^1 \varphi(y)\de y,
\end{equation}
and in particular, choosing \(\varphi=\mathbf 1_{[0,v]}\) for \(v\in[0,1]\) gives
\begin{equation}
\label{eq:h-marginal}
    \int_0^1 h_v(t)\de t=v,
    \qquad 0\leq v\leq1.
\end{equation}
For probability measures \(\mu,\nu\) on \([0,1]\), write
\(\mu\preceq_{\rm st}\nu\) if
\(
    F_\mu(y)\geq F_\nu(y)
    \text{ for all }y\in[0,1],
\)
where \(F_\mu,F_\nu\) are the corresponding distribution functions.
The stochastic-increasing condition used below is the copula-kernel formulation of positive regression dependence, see \cite{lehmann1966dependence}.
For the copula-theoretic formulations of SI, LTD, and RTI used here, we follow \cite[Ch.~5]{nelsen2006introduction}.
We call a copula \(C\) \emph{stochastically increasing} (\emph{SI}) if its Markov kernel can be chosen such that, for every \(v\in(0,1)\), the map
\begin{equation}\label{eq:si_definition}
    t\mapsto K_C(t,[0,v])
\end{equation}
is non-increasing.
Similarly, \(C\) is called \emph{stochastically decreasing} (\emph{SD}) if its Markov kernel can be chosen such that, for every \(v\in(0,1)\), the map
\(
    t\mapsto K_C(t,[0,v])
\)
is non-decreasing.
Throughout, whenever \(C\) is SI or SD, we fix such a monotone
version of \(K_C\).

We also use the following classical tail-monotonicity notions.
\(C\) is \emph{left-tail decreasing} (\emph{LTD}) if, for every \(v\in(0,1)\),
\begin{equation}\label{eq:ltd_definition}
    u\longmapsto \frac{C(u,v)}{u},
    \qquad u\in(0,1],
\end{equation}
is non-increasing.
Further, \(C\) is \emph{right-tail increasing} (\emph{RTI}) if, for every
\(v\in(0,1)\),
\begin{equation}\label{eq:rti_definition}
    u\longmapsto \frac{v-C(u,v)}{1-u},
    \qquad u\in[0,1),
\end{equation}
is non-increasing, cf.~\cite[Thm.~5.2.5]{nelsen2006introduction}.
It is classical that a stochastically increasing copula is necessarily both LTD and RTI, see \cite[Thm.~5.2.12]{nelsen2006introduction}, so the assumption in Theorem~\ref{thm:main} is stronger than the assumption in Theorem~\ref{thm:ltd-rti-xi-rho}.
Similarly to \eqref{eq:ltd_definition} and \eqref{eq:rti_definition}, one may also define the opposite tail monotonicity conditions.
\(C\) is \emph{left-tail increasing} (\emph{LTI}) if, for every \(v\in(0,1)\),
\(
    u\mapsto \frac{C(u,v)}{u},
\)
is non-decreasing for \(u\in(0,1]\), and \(C\) is \emph{right-tail decreasing} (\emph{RTD}) if, for every
\(v\in(0,1)\),
\(
    u\mapsto \frac{v-C(u,v)}{1-u},
\)
is non-decreasing for \(u\in[0,1)\).

We now recall the three rank coefficients considered in this paper.
Spearman's rho, as introduced in \cite{spearman1904proof}, is a classical measure of concordance.
It quantifies monotone positive or negative dependence and takes values in \([-1,1]\), with the extremal values attained at the Fréchet bounds \(M\) and \(W\).
If \((X,Y)\) has continuous marginal distribution functions, then Spearman's rho is the Pearson correlation of the random vector \((F(X),G(Y))\), where \(F\) and \(G\) are the marginal distribution functions of \(X\) and \(Y\), respectively.
If $(X,Y)$ has copula \(C\), then Spearman's rho can be expressed in terms of \(C\) as
\begin{equation}
\label{eq:rho-definition}
\rho(C)
:=
12\int_0^1\int_0^1 C(u,v)\de u\de v-3
=
12\int_0^1\int_0^1 \bigl(C(u,v)-uv\bigr)\de u\de v,
\end{equation}
see, e.g., \cite[Thm.~5.1.6]{nelsen2006introduction}.

Alongside Spearman's rho, Kendall's tau is one of the most prominent rank correlation coefficients, which has a strikingly simple probabilistic meaning.
We say that two pairs of observations \((X_1,Y_1)\) and \((X_2,Y_2)\) are \emph{concordant} if the ranks of both coordinates agree, i.e., if either \(X_1 < X_2\) and \(Y_1 < Y_2\) or \(X_1 > X_2\) and \(Y_1 > Y_2\).
In contrast, they are \emph{discordant} if the ranks disagree, i.e., if either \(X_1 < X_2\) and \(Y_1 > Y_2\) or \(X_1 > X_2\) and \(Y_1 < Y_2\).
Let $(X,Y)$ and $(X',Y')$ have the same distribution and be independent.
Then Kendall's tau is defined as the difference between the probabilities of concordance and discordance, see \cite{kendall1938new} or also \cite[Ch.~5]{nelsen2006introduction}:
\[
    \tau(X,Y)
    :=\Prob\left[(X - X')(Y - Y') > 0\right] - \Prob\left[(X - X')(Y - Y') < 0\right].
\]
Let \(C\) be a bivariate copula with conditional laws \((H_t)_{t\in[0,1]}\) as above.
For \(t,s\in[0,1]\), define
\begin{equation}\label{eq:local_discordance}
    D(t,s)
    :=
    (H_t\otimes H_s)\{(x,y)\in[0,1]^2:x>y\},
\end{equation}
where \(H_t\otimes H_s\) is the product measure of \(H_t\) and \(H_s\).
Equivalently, \(D(t,s)=\Prob(Y_t>Y_s)\) if \(Y_t\sim H_t\) and \(Y_s\sim H_s\) are independent, so $D(t,s)$ is the pairwise discordance probability of the conditional laws \(H_t\) and \(H_s\).
If \((U,V)\) and \((U',V')\) are independent with copula \(C\), then the
two discordant events
\[
    \{U<U',\,V>V'\}
    \qquad\text{and}\qquad
    \{U>U',\,V<V'\}
\]
have the same probability.
Further, by disintegrating both independent copies, the joint law of
\((U,U',V,V')\) is
\(
    \de t\,\de s\,H_t(\de x)\,H_s(\de y).
\)
Therefore,
\(
    \Prob(U<U',\,V>V')
    =
    \int_{0<t<s<1} D(t,s)\de t\de s
\)
by the definition of \(D\).
Since the marginals are continuous, ties have probability zero, and thus
\begin{equation}
\label{eq:tau-d-relation}
    \tau(C)
    =
    1-2\Prob((U-U')(V-V')<0)
    =
    1-4\int_{0<t<s<1}D(t,s)\de t\de s.
\end{equation}

Chatterjee's rank correlation measures a different aspect of dependence. It is directional and quantifies the degree to which the second coordinate is determined by the first one, rather than the strength of monotone association.
\(\xi\) takes values in \([0,1]\), with value \(0\) corresponding to independence and value \(1\) corresponding to perfect functional dependence of the second coordinate on the first.
In particular, unlike \(\rho\) and \(\tau\), it does not distinguish positive from negative monotone dependence and also detects non-monotone functional relationships.
For a copula \(C\), Chatterjee's rank correlation is given by
\begin{equation}
\label{eq:xi-definition}
\xi(C)
\coloneq
6 \int_0^1\int_0^1 \left(\partial_1 C(u,v)\right)^2 \de u\de v - 2 
=
1-6\int_0^1\int_0^1 h_v(t)(1-h_v(t))\de t\de v,
\end{equation}
see \cite{chatterjee2020,dette2013copula}, where the second equality follows immediately from \eqref{eq:h-marginal}.

\section{Order violations under stochastic order}
\label{sec:pairwise}

We now prove the local inequality underlying the Kendall bound.
The result is a sharp order-violation bound for two stochastically ordered distribution functions.
It estimates the order-violation probability \(\Prob(X>Z)\) by a symmetric dispersion functional of the two distribution functions.
When applied with \(X\sim H_t\) and \(Z\sim H_s\), this probability becomes the local discordance term \(D(t,s)\) in the representation \eqref{eq:tau-d-relation}.
Thus Theorem~\ref{thm:pairwise} provides the pointwise estimate whose integrated version we shall use to prove the Kendall bound in Theorem~\ref{thm:main}.

\begin{theorem}[Sharp order-violation bound under stochastic order]
\label{thm:pairwise}
Let \(F\) and \(G\) be continuous distribution functions on \([0,1]\), satisfying
\(
    F(0)=G(0)=0
\)
and
\(
    F(1)=G(1)=1,
\)
let \(F\) be strictly increasing, and assume
\[
    F(y)\geq G(y)
    \qquad
    \text{for all }y\in[0,1].
\]
Let \(X\sim F\) and \(Z\sim G\) be independent.
Then
\begin{equation}
\label{eq:pairwise-bound}
    \Prob(X>Z)
    \leq
    \frac32
    \left\{
        \int_0^1 F(y)(1-F(y))\de G(y)
        +
        \int_0^1 G(y)(1-G(y))\de F(y)
    \right\}.
\end{equation}
\end{theorem}

The bound is independent of the copula setting.
If \(F\geq G\), then \(X\sim F\) is stochastically no larger than \(Z\sim G\), and the event \(\{X>Z\}\) is the pairwise order violation of this stochastic ordering.
Thus \eqref{eq:pairwise-bound} gives a distribution-free control of a natural misranking probability under stochastic dominance, in terms of the symmetrized cross-rank variance functional
\[
    \E[F(Z)(1-F(Z))]
    +
    \E[G(X)(1-G(X))].
\]
The terminology is literal: if \(Y\sim F\) is independent of \(Z\) and \(W\sim G\) is independent of \(X\), then
\[
    F(Z)(1-F(Z))
    =
    \operatorname{Var}\!\left(\mathbf 1_{\{Y\leq Z\}}\mid Z\right),
    \qquad
    G(X)(1-G(X))
    =
    \operatorname{Var}\!\left(\mathbf 1_{\{W\leq X\}}\mid X\right).
\]
Thus the right-hand side averages the Bernoulli variances of the two cross-rank comparisons, namely the rank of a \(G\)-sample on the \(F\)-scale and the rank of an \(F\)-sample on the \(G\)-scale.
The factor \(3/2\) is optimal, since equality holds, e.g., when \(F=G\).
Below we use this standalone inequality with \(F=h_\cdot(t)\) and \(G=h_\cdot(s)\), but no copula structure is used in its proof.

For the proof, we first isolate a simple deterministic inequality for monotone maps below the identity in the following lemma.

\begin{lemma}
\label{lem:monotone-eta}
Let \(\eta\colon[0,1]\to[0,1]\) be non-decreasing and satisfy
\(
    \eta(t)\leq t
    \text{ for all }t\in[0,1].
\)
Then
\begin{equation}
\label{eq:eta-ineq}
    \int_0^1 (t-\eta(t))^2\de t
    \leq
    \frac23\int_0^1 (t-\eta(t))\de t.
\end{equation}
\end{lemma}

The assumptions in Lemma~\ref{lem:monotone-eta} are tailored to the stochastic order between the conditional laws that appear in the SI application in Section~\ref{sec:proof}.
There, for \(t<s\), we take \(F(y):=h_y(t)\) and \(G(y):=h_y(s)\).
Since the copula is stochastically increasing, the chosen kernel satisfies \(h_y(t)\geq h_y(s)\) for all \(y\in[0,1]\), and hence \(F\geq G\).
The quantile change of variables \(\eta(r):=G(F^{-1}(r))\), \(0\leq r\leq1\), then produces a non-decreasing map satisfying \(\eta(r)=G(F^{-1}(r))\leq F(F^{-1}(r))=r\).
Thus \(r-\eta(r)\) measures, in the \(F\)-quantile scale, the separation between the two conditional distribution functions.

\begin{figure}[t!]
  \centering
  \includegraphics[width=0.95\textwidth]{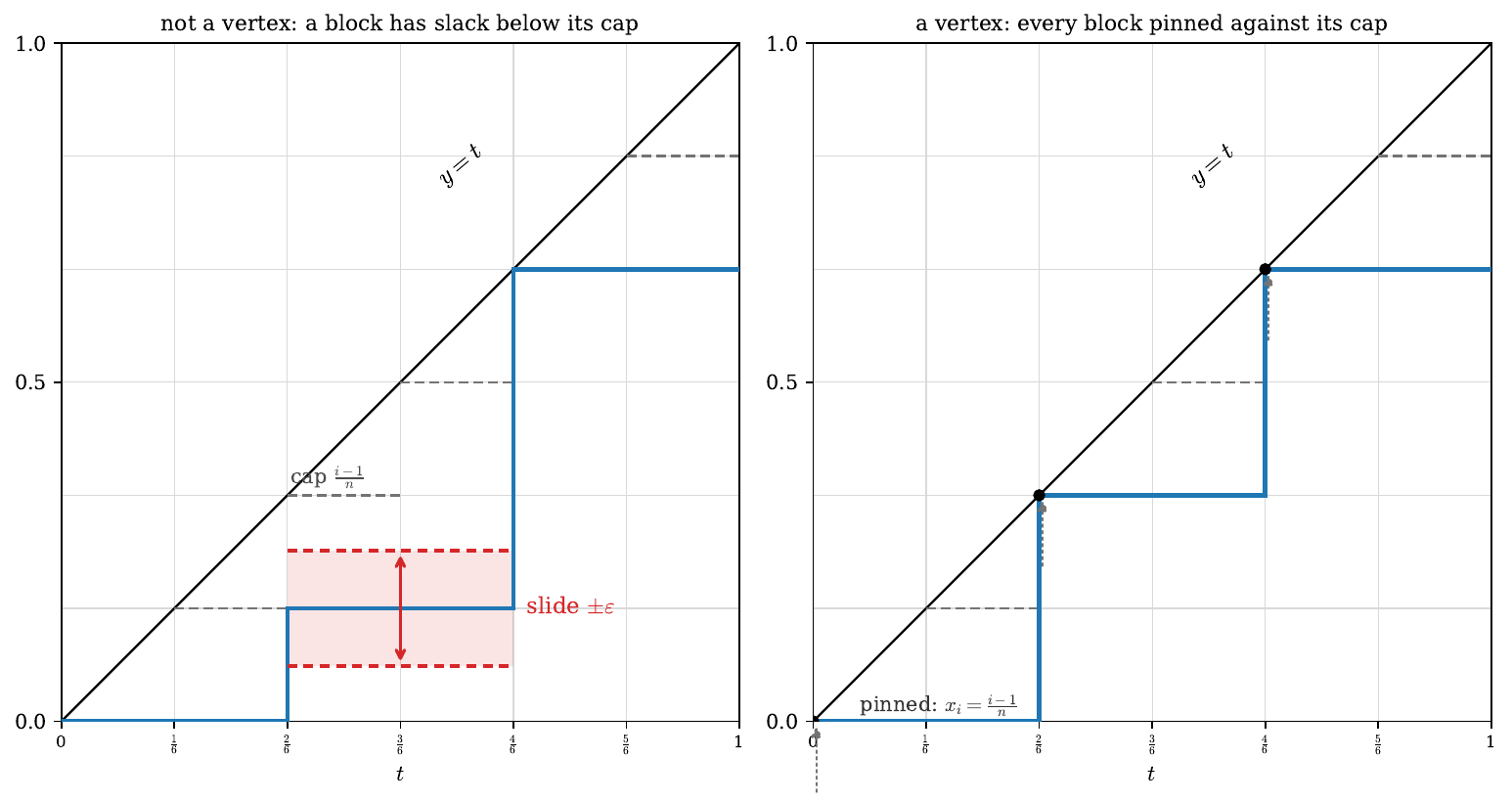}
  \caption{The vertices of the feasible polytope in the proof of Lemma~\ref{lem:monotone-eta}, illustrated for \(n=6\).
  Feasible grid step functions \(\eta\equiv x_i\) on \([(i-1)/n,i/n)\) (blue) obey the monotonicity constraints \(0\leq x_1\leq\cdots\leq x_n\) and the caps \(x_i\leq(i-1)/n\) in \eqref{eq:eta-proof-grid-constraints}.
  The dashed ceiling marks the caps, which coincide with the diagonal \(y=t\) at the left edge of each cell.
  Left: if a maximal constant block lies strictly below its cap, the whole block can be shifted up or down by a small \(\varepsilon\) (red) while remaining feasible, so \(\eta\) is the midpoint of two distinct feasible points and hence not a vertex.
  Right: at a vertex every maximal block is pinned against its cap, that is, its value equals the left endpoint \((p-1)/n\) (black dots on the diagonal), so it cannot be raised and no two-sided perturbation exists.
  The convex functional \(Q\) of \eqref{eq:eta-proof-Q-def} therefore attains its maximum at such a pinned staircase, for which \eqref{eq:eta-ineq} is verified blockwise.}
  \label{fig:monotone-eta}
\end{figure}

\begin{proof}[Proof of Lemma~\ref{lem:monotone-eta}]
Put
\begin{equation}
\label{eq:eta-proof-Q-def}
    Q(\eta)
    :=
    \int_0^1\left\{(t-\eta(t))^2-\frac23(t-\eta(t))\right\}\de t,
\end{equation}
so we must prove \(Q(\eta)\le0\).
We first consider step functions on the grid \(0,1/n,\ldots,1\), say \(\eta(t)=x_i\) on \([(i-1)/n,i/n)\), under the constraints
\begin{equation}
\label{eq:eta-proof-grid-constraints}
    0\le x_1\le\cdots\le x_n,
    \qquad
    x_i\le \frac{i-1}{n}.
\end{equation}
Let
\[
    P_n
    :=
    \left\{
        x\in\R^n:
        0\leq x_1\leq\cdots\leq x_n,
        \quad
        x_i\leq \frac{i-1}{n},\ i=1,\ldots,n
    \right\}
\]
be the feasible set in \eqref{eq:eta-proof-grid-constraints}.
Each constraint defining \(P_n\) is a linear inequality in the variable \(x=(x_1,\ldots,x_n)\): for instance, \(x_i\leq x_{i+1}\) is the halfspace condition \(x_i-x_{i+1}\leq0\).
Hence \(P_n\) is a finite intersection of closed halfspaces. Since also
\[
    0\leq x_i\leq x_n\leq\frac{n-1}{n}
    \qquad(i=1,\ldots,n),
\]
the set \(P_n\) is bounded. Thus \(P_n\) is an \(H\)-polytope in the sense of \cite[Def.~0.1]{ziegler1995polytopes}, and in particular it is compact.
By \cite[Def.~2.1 and Prop.~2.2 (i)]{ziegler1995polytopes}, its vertices are the zero-dimensional faces of \(P_n\), and
\(
    P_n=\operatorname{conv}(\operatorname{vert}(P_n)),
\)
i.e.~$P_n$ is the convex hull of its vertices.
Equivalently for the argument below, a vertex is an extreme point: it cannot be written as the midpoint of two distinct points of \(P_n\).
Since \(Q\) is convex in \((x_1,\ldots,x_n)\), its maximum over \(P_n\) is attained at a vertex. The structure of these vertices is illustrated in Figure~\ref{fig:monotone-eta}.

Let \(x\) be a vertex, and let \(\{p,p+1,\ldots,q\}\) be a maximal block on which the coordinates of \(x\) are constant, say
\begin{equation}
\label{eq:eta-proof-constant-block}
    x_p=x_{p+1}=\cdots=x_q=c.
\end{equation}
We claim that \(c=(p-1)/n\).
If \(p=1\), then
\(
    0\leq x_1\leq 0,
\)
so \(c=0=(p-1)/n\).
Suppose \(p>1\).
By maximality of the block in \eqref{eq:eta-proof-constant-block}, \(x_{p-1}<c\), and, if \(q<n\), also \(c<x_{q+1}\).
If \(c<(p-1)/n\), then we can choose \(\varepsilon>0\) so small that replacing all coordinates in this block by \(c+\varepsilon\), or by \(c-\varepsilon\), preserves all inequalities in \eqref{eq:eta-proof-grid-constraints}.
Indeed, the upper constraints are preserved because
\[
    c+\varepsilon\leq \frac{p-1}{n}\leq \frac{i-1}{n},
    \qquad i=p,\ldots,q,
\]
and the monotonicity constraints are preserved by the choice of \(\varepsilon\).
The two modified vectors are distinct feasible points whose midpoint is \(x\), contradicting that \(x\) is a vertex.
Hence \(c=(p-1)/n\).
Thus every vertex is obtained by partitioning \([0,1]\) into grid intervals \([a,b)\) on which \(\eta(t)=a\).
The contribution of such a block is
\begin{equation}
\label{eq:eta-proof-block-contribution}
    \int_a^b\left\{(t-a)^2-\frac23(t-a)\right\}\de t
    =
    \frac{(b-a)^2}{3}\bigl((b-a)-1\bigr)
    \le0.
\end{equation}
Summing \eqref{eq:eta-proof-block-contribution} over the blocks gives \(Q(\eta)\le0\) for all feasible grid step functions.

For a general non-decreasing \(\eta\le\mathrm{id}\), approximate from the left by the grid step functions
\(
    \eta_k(t)=\eta\left(\frac{i-1}{k}\right)
\)
on
\(
    t\in\left[\frac{i-1}{k},\frac{i}{k}\right).
\)
Then \(\eta_k(t)\to\eta(t)\) at every continuity point of \(\eta\), hence almost everywhere, and dominated convergence applied to \eqref{eq:eta-proof-Q-def} yields \(Q(\eta)=\lim_{k\to\infty} Q(\eta_k)\le0\).
\end{proof}


\begin{proof}[Proof of Theorem~\ref{thm:pairwise}]
Since \(F\) is continuous and strictly increasing with \(F(0)=0\) and \(F(1)=1\), it has a continuous inverse \(F^{-1}\colon[0,1]\to[0,1]\).
Define \(\eta(t):=G(F^{-1}(t))\) for \(0\leq t\leq1\).
Then \(\eta\) is non-decreasing and, since \(F\geq G\), it satisfies \(\eta(t)=G(F^{-1}(t))\leq F(F^{-1}(t))=t\), so \(\eta\) fulfills the hypotheses of Lemma~\ref{lem:monotone-eta}.

Set
\(
    U:=F(X).
\)
By the regularity assumptions, \(U\sim{\rm Uniform}(0,1)\) and \(X=F^{-1}(U)\).
Since \(X\) and \(Z\) are independent, \(U\) and \(F(Z)\) are independent.
For \(0\leq t\leq1\),
\[
    \Prob(F(Z)\leq t)=\Prob(Z\leq F^{-1}(t))=G(F^{-1}(t))=\eta(t),
\]
so \(F(Z)\) has distribution function \(\eta\).
Further,
\(
    \Prob(X>Z)=\Prob(F(X)>F(Z)),
\)
since \(F\) is strictly increasing.
With \(U:=F(X)\), we have \(U\sim{\rm Uniform}(0,1)\), and \(U\) is independent
of \(F(Z)\). Therefore, conditioning on \(F(Z)\),
\[
    \Prob(F(X)>F(Z))
    =
    \E[\Prob(U>F(Z)\mid F(Z))]
    =
    \E[1-F(Z)].
\] 
Since \(F(Z)\) has distribution function \(\eta\), it follows
\(
    \E[1-F(Z)]
    =
    \int_0^1 \Prob(F(Z)\leq t)\de t
    =
    \int_0^1 \eta(t)\de t,
\)
and hence
\begin{equation}
\label{eq:PXZ-eta}
    \Prob(X>Z)=\int_0^1 \eta(t)\de t.
\end{equation}
Next, since $Z\sim G$, it is
\begin{equation}\label{eq:dispersion-interpretation}
    \int_0^1 F(y)(1-F(y))\de G(y)
    =
    \E[F(Z)(1-F(Z))].
\end{equation}
Further, since \(F(Z)\) has distribution function \(\eta\), it is
\begin{equation}\label{eq:eta-distribution}
    \E[F(Z)(1-F(Z))]
    = \int_0^1 t(1-t)\,\mu_\eta(dt)
    =
    \int_0^1 t(1-t)\de \eta(t),
\end{equation}
where $\mu_\eta$ is the probability measure on $[0,1]$ with distribution function $\eta$.
The function \(\eta\) is continuous and of bounded variation, while \(\phi(t)=t(1-t)\) is continuously differentiable and satisfies \(\phi(0)=\phi(1)=0\).
Therefore Stieltjes integration by parts as in \cite[Thm.~18.4]{billingsley2012probability} gives
\[
    \int_0^1 t(1-t)\de \eta(t)
    =
    -\int_0^1 \eta(t)(1-2t)\de t
    =
    \int_0^1 \eta(t)(2t-1)\de t.
\]
Combining this with \eqref{eq:dispersion-interpretation} and \eqref{eq:eta-distribution} yields
\begin{equation}
\label{eq:first-dispersion-eta}
    \int_0^1 F(y)(1-F(y))\de G(y)
    =
    \int_0^1 \eta(t)(2t-1)\de t.
\end{equation}
Similarly,
\(
    \int_0^1 G(y)(1-G(y))\de F(y)
    =
    \E[G(X)(1-G(X))].
\)
Since \(X=F^{-1}(U)\) almost surely, we have \(G(X)=G(F^{-1}(U))=\eta(U)\) almost surely.
Further, since \(U\sim\mathrm{Uniform}(0,1)\), it follows that
\begin{equation}
\label{eq:second-dispersion-eta}
    \int_0^1 G(y)(1-G(y))\de F(y)
    =
    \E[\eta(U)(1-\eta(U))]
    =
    \int_0^1 \eta(t)(1-\eta(t))\de t.
\end{equation}
Adding \eqref{eq:first-dispersion-eta} and \eqref{eq:second-dispersion-eta}, we get
\[
    \int_0^1 F(y)(1-F(y))\de G(y)
    +
    \int_0^1 G(y)(1-G(y))\de F(y)
    =
    \int_0^1 \eta(t)(2t-\eta(t))\de t.
\]
Recalling also \eqref{eq:PXZ-eta}, the desired inequality \eqref{eq:pairwise-bound} is therefore equivalent to
\begin{equation}
\label{eq:eta-inequality}
    \int_0^1 \eta(t)\de t
    \leq
    \frac32
    \int_0^1 \eta(t)(2t-\eta(t))\de t.
\end{equation}
Put
\(
    \delta(t):=t-\eta(t).
\)
Then \(\delta(t)\geq0\), and
\(
    \int_0^1 \eta(t)\de t
    =
    \frac12-\int_0^1\delta(t)\de t,
\)
while
\[
    \int_0^1 \eta(t)(2t-\eta(t))\de t
    =
    \int_0^1 (t-\delta(t))(t+\delta(t))\de t
    =
    \frac13-\int_0^1\delta(t)^2\de t.
\]
The desired inequality \eqref{eq:eta-inequality} is therefore equivalent to
\(
    \int_0^1 \delta(t)^2\de t
    \leq
    \frac23\int_0^1\delta(t)\de t,
\)
which is precisely Lemma~\ref{lem:monotone-eta}.
\end{proof}

\section{From order violations to the Kendall bound}
\label{sec:proof}

For the Kendall bound from Theorem~\ref{thm:main}, we will make use of the sharp order-violation bound from Theorem~\ref{thm:pairwise} pointwise for the conditional laws of the Markov kernel.
First, combining \eqref{eq:tau-d-relation} and \eqref{eq:xi-definition}, the inequality \(\xi(C)\leq\tau(C)\) is equivalent to
\begin{equation}
\label{eq:d-b-target}
    \int_{0<t<s<1}D(t,s)\de t\de s
    \leq
    \frac32
    \int_0^1\int_0^1 h_v(t)(1-h_v(t))\de t\de v.
\end{equation}
Hence, it is enough to prove \eqref{eq:d-b-target} for every SI copula \(C\) with conditional laws \((H_t)_{t\in[0,1]}\) and \(h_v(t) = H_t([0,v])\), as defined in \eqref{eq:h-definitions}.
We first prove \eqref{eq:d-b-target} under a kernel-regularity assumption, so that Theorem~\ref{thm:pairwise} can be applied directly.
The regularity assumption is then removed by a Gaussian-copula regularization argument.

\begin{definition}[Kernel regularity]
\label{def:regular}
    We call the chosen Markov kernel \(K_C\) \emph{kernel-regular} if, for
    \(\lambda\)-a.e.~\(t\), the conditional distribution function
    \(
        v\mapsto h_v(t)=K_C(t,[0,v])
    \)
    is continuous on \([0,1]\), strictly increasing in \(v\), and satisfies
    \(
        h_0(t)=0
        \text{ and }
        h_1(t)=1.
    \)
\end{definition}

\begin{proposition}
\label{prop:regular-proof}
If \(C\) is an SI copula whose chosen SI Markov kernel is kernel-regular,
then
\(
    \xi(C)\leq\tau(C).
\)
\end{proposition}

\begin{proof}
    Let \(K_C\) be the fixed SI Markov kernel of \(C\), with \(H_t\) and \(h_y(t)\) as in \eqref{eq:h-definitions}.
    Let \(T\subseteq[0,1]\) be a full-measure set such that, for every \(t\in T\), the map \(y\mapsto h_y(t)\) is continuous, strictly increasing, and satisfies \(h_0(t)=0\) and \(h_1(t)=1\).
    If \(t,s\in T\) and \(t<s\), stochastic increasingness of the chosen kernel
    gives
    \begin{equation}
    \label{eq:regular-proof-si-order}
        h_y(t)\geq h_y(s)
        \qquad
        \text{for all }y\in(0,1).
    \end{equation}
    Together with the endpoint identities, this holds for all \(y\in[0,1]\).
    Hence \(H_t\preceq_{\rm st}H_s\).

    For such \(t<s\), apply Theorem~\ref{thm:pairwise} with \(F(y)=h_y(t)\) and \(G(y)=h_y(s)\).
    The regularity assumption ensures that \(F\) and \(G\) are continuous, strictly increasing distribution functions on \([0,1]\), with \(F(0)=G(0)=0\) and \(F(1)=G(1)=1\).
    By \eqref{eq:regular-proof-si-order}, we also have \(F\geq G\).
    Hence, since \(D(t,s)=\Prob(X>Z)\) for independent \(X\sim H_t\) and \(Z\sim H_s\), Theorem~\ref{thm:pairwise} gives
    \begin{equation}
    \label{eq:regular-proof-pairwise-bound}
        D(t,s)
        \leq
        \frac32
        \left\{
            b_{t,s}+b_{s,t}
        \right\},
    \end{equation}
    where
    \begin{equation}
    \label{eq:regular-proof-bts-def}
        b_{t,s}
        :=
        \int_0^1 h_y(t)(1-h_y(t))\,H_s(\de y).
    \end{equation}
    The maps \((t,s)\mapsto D(t,s)\) and \((t,s)\mapsto b_{t,s}\) are
    measurable by the standard product and integration properties of kernels
    \cite[Lems.~3.2--3.3]{kallenberg2021foundations}. Since the exceptional
    set outside \(T^2\) has two-dimensional Lebesgue measure zero, integrating
    \eqref{eq:regular-proof-pairwise-bound} over the triangle gives
    \begin{equation}
    \label{eq:regular-proof-triangle-bound}
        \int_{0<t<s<1}D(t,s)\de t\de s
        \leq
        \frac32
        \int_{0<t<s<1}
        \left\{
            b_{t,s}+b_{s,t}
        \right\}
        \de t\de s.
    \end{equation}
    Moreover,
    \begin{equation}
    \label{eq:regular-proof-triangles-cover-square}
        \int_{0<t<s<1} b_{t,s}\de t\de s
        +
        \int_{0<t<s<1} b_{s,t}\de t\de s
        =
        \int_0^1\int_0^1 b_{t,s}\de s\de t,
    \end{equation}
    because the two triangular regions cover the square up to the diagonal.
    By the mixture identity \eqref{eq:mixture-uniform-functions}, the definition
    of \(b_{t,s}\) in \eqref{eq:regular-proof-bts-def} gives
    \begin{equation}
    \label{eq:regular-proof-reconstruct-xi-term}
        \int_0^1\int_0^1 b_{t,s}\de s\de t
        =
        \int_0^1\int_0^1
            h_y(t)(1-h_y(t))\de y\de t.
    \end{equation}
    Combining \eqref{eq:regular-proof-triangle-bound},
    \eqref{eq:regular-proof-triangles-cover-square}, and
    \eqref{eq:regular-proof-reconstruct-xi-term} gives precisely
    \eqref{eq:d-b-target}, and thus the proof is complete.
\end{proof}

In order to prove the full statement of Theorem~\ref{thm:main}, we record a sharpness example for the Kendall bound under SI.
The family is illustrated in Figure~\ref{fig:ordinal-sums}.

\begin{figure}[t!]
  \centering
  \includegraphics[width=\textwidth]{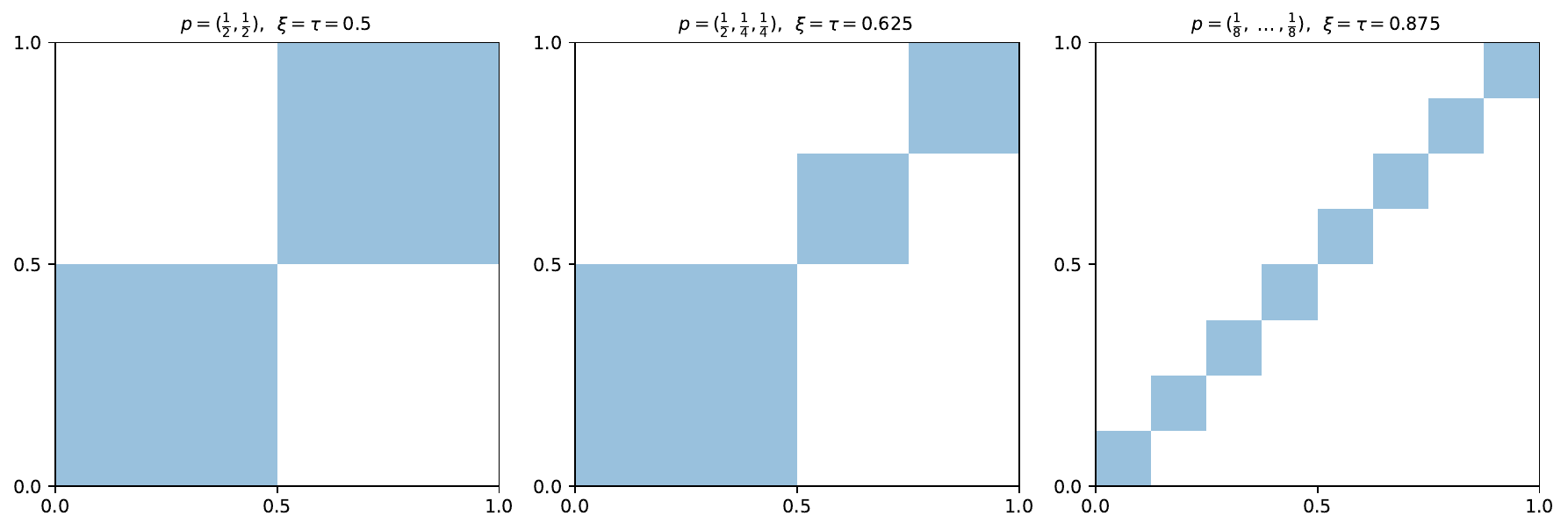}
  \caption{Supports of ordinal sums of product copulas for three partitions of \([0,1]\) with interval lengths \(p_j\), the equality family of Proposition~\ref{prop:ordinal-sum}: the mass is distributed uniformly over the union of the shaded squares \(I_j^2\), and \(\xi(C)=\tau(C)=1-\sum_j p_j^2\) in each case.
  Refining the partition increases both coefficients towards \(1\), while the trivial partition gives \(C=\Pi\) with \(\xi=\tau=0\).
  The family thus interpolates the equality \(\xi=\tau\) across the whole admissible range.}
  \label{fig:ordinal-sums}
\end{figure}

\begin{proposition}[Ordinal sums of product copulas]
\label{prop:ordinal-sum}
Let \(\{I_j\}_{j\in J}\) be a finite or countable family of pairwise disjoint
non-degenerate intervals such that \(\lambda([0,1]\setminus\bigcup_j I_j)=0\), and put \(p_j=\lambda(I_j)>0\).
Consider the copula obtained as follows: choose \(j\) with probability \(p_j\), and conditionally on \(j\), let \(U\) and \(V\) be independent uniform random variables on the same interval \(I_j\).
Then \(C\) is SI and
\[
    \xi(C)=\tau(C)=1-\sum_{j\in J}p_j^2.
\]
\end{proposition}

\begin{proof}
    We take a disjoint Borel version of the interval partition, changing only endpoints if necessary.
At common endpoints of adjacent intervals, the conditional law may be chosen arbitrarily, for instance according to the interval on the right.
The set of such endpoints is countable and \(\lambda\)-null, so this convention does not affect the copula, \(\xi\), or \(\tau\).

If \(t\in I_j\), the conditional law \(H_t\) is uniform on \(I_j\).
Hence, for \(t<s\), either \(t\) and \(s\) lie in the same interval and the conditional laws are identical, or \(t\) lies in an interval strictly to the left of the interval containing \(s\).
Since intervals in the partition are ordered along \([0,1]\), this means that for every \(v\in(0,1)\) the map \(t\mapsto h_v(t)\) is non-increasing.
Thus \(C\) is SI.
Furthermore, the map \(v\mapsto h_v(t)\) is the distribution function of the
uniform law on \(I_j\). Let
\(
    a_j:=\inf I_j,
    b_j:=\sup I_j,
    \text{ and }
    p_j=b_j-a_j.
\)
Endpoint conventions are immaterial for the following integral identities.
For \(t\in I_j\), we may write, up to values at the endpoints,
\[
h_v(t)
=
\begin{cases}
0, & v<a_j,\\
\dfrac{v-a_j}{p_j}, & a_j<v<b_j,\\
1, & v>b_j.
\end{cases}
\]
Consequently,
\[
\int_0^1 h_v(t)(1-h_v(t))\de v
=
\int_{a_j}^{b_j}
\frac{v-a_j}{p_j}
\left(1-\frac{v-a_j}{p_j}\right)
\de v 
=
p_j\int_0^1 x(1-x)\de x
=
\frac{p_j}{6}.
\]
Integrating over \(t\in I_j\) gives the contribution
\(
\int_{I_j}\int_0^1 h_v(t)(1-h_v(t))\de v\de t
=
\frac{p_j^2}{6}.
\)
Summing over \(j\in J\) and using \eqref{eq:xi-definition} yields
\(
\xi(C)
=
1-6\sum_{j\in J}\frac{p_j^2}{6}
=
1-\sum_{j\in J}p_j^2.
\)

For Kendall's tau, take two independent copies \((U,V)\) and \((U',V')\) of the random vector with copula \(C\).
Let \(L\) and \(L'\) denote the random indices of the intervals containing \((U,V)\) and \((U',V')\), respectively.
If \(L\neq L'\), then the two observations lie in distinct diagonal blocks.
Hence the order of the \(U\)-coordinates is the same as the order of the \(V\)-coordinates, so the two observations are concordant.
Thus the conditional Kendall contribution on \((L\neq L')\) is \(1\).
On the other hand, conditional on \(L=L'=j\), both observations lie in the same block \(I_j^2\), where the copula is the product copula.
Equivalently, inside this block the signs of \(U-U'\) and \(V-V'\) are independent and symmetric.
Therefore the conditional probabilities of concordance and discordance are both \(1/2\), and the conditional Kendall contribution is zero.
Since
\(
\Prob(L=L')=\sum_{j\in J}\Prob(L=j)^2=\sum_{j\in J}p_j^2,
\)
we obtain \(\tau(C)=\Prob(L\neq L')=1-\sum_{j\in J}p_j^2\).
Thus \(\xi(C)=\tau(C)\).
\end{proof}

We are now ready to prove the main theorem in full generality.

\begin{proof}[Proof of Theorem~\ref{thm:main}]
Let \(C\) be SI.
By Proposition~\ref{prop:regular-proof}, if the chosen SI Markov kernel of \(C\) is  kernel-regular, then \(\xi(C)\leq\tau(C)\).
It hence remains to remove the regularity assumption.

Let \(G_r\), \(0<r<1\), denote the Gaussian copula with correlation parameter \(r\), and put
\[
    C_r:=C*G_r,
\]
where \(*\) denotes the Markov product from \eqref{eq:markov-product-kernel}.
Let \(\Phi\) denote the standard normal distribution function.
We use the following version of the Gaussian transition kernel, obtained from the standard conditional distribution formula for the bivariate normal copula, see \cite[Eq.~(3.1)]{meyer2013bivariate}.
For \(x\in(0,1)\),
\[
    K_r(x,[0,y])
    =
    \Phi\left(
        \frac{\Phi^{-1}(y)-r\Phi^{-1}(x)}{\sqrt{1-r^2}}
    \right),
    \qquad 0\leq y\leq1
\]
with endpoint conventions
\(
    K_r(0,\cdot)=\delta_0
\)
and
\(
    K_r(1,\cdot)=\delta_1.
\)
The endpoint values are immaterial for the Gaussian copula itself, but they make \(x\mapsto K_r(x,[0,y])\) non-increasing on all of \([0,1]\) for every \(y\in(0,1)\).
By the kernel-composition formula \eqref{eq:markov-product-kernel}, \(C_r=C*G_r\) admits the Markov kernel
\[
    K_{C_r}(t,E)
    =
    \int_0^1 K_r(x,E)\,K_C(t,\de x),
    \qquad
    t\in[0,1],\quad E\in\mathcal B([0,1]).
\]

We first check that this kernel is stochastically increasing.
Since \(G_r\) is stochastically increasing, see, e.g.,
\cite[Table~5]{ansari2023dependence}, the map
\(
    x\mapsto K_r(x,[0,y])
\)
is non-increasing for every \(y\in(0,1)\).
Moreover, since \(K_C\) is the fixed SI kernel of \(C\), we have
\(K_C(t,\cdot)\preceq_{\rm st}K_C(s,\cdot)\) whenever \(t<s\).
By the standard test-function characterization of stochastic order as in \cite[Sec.~1.A, Eq.~(1.A.7)]{shaked2007stochastic}, if \(\mu\preceq_{\rm st}\nu\), then integration against bounded non-increasing test functions is larger under \(\mu\) than under \(\nu\).
Therefore, for
\(t<s\),
\[
    K_{C_r}(t,[0,y])
    =
    \int_0^1 K_r(x,[0,y])\,K_C(t,\de x)
    \geq
    \int_0^1 K_r(x,[0,y])\,K_C(s,\de x)
    =
    K_{C_r}(s,[0,y]).
\]
Thus \(K_{C_r}\) is a monotone Markov kernel of \(C_r\), and hence \(C_r\)
is stochastically increasing.

We next note that this kernel is kernel-regular in the sense of Definition~\ref{def:regular}.
For \(\lambda\)-a.e.~\(t\), the conditional distribution function of the second coordinate under \(C_r\), given the first coordinate \(t\), is
\begin{equation}
\label{eq:main-proof-regularized-h}
    h_y^{(r)}(t)
    =
    K_{C_r}(t,[0,y])
    =
    \int_0^1 K_r(x,[0,y])\,K_C(t,\de x).
\end{equation}
For each \(x\in(0,1)\), the map \(y\mapsto K_r(x,[0,y])\) is continuous on \([0,1]\), strictly increasing on \((0,1)\), and has a strictly positive density on \((0,1)\).
Since \eqref{eq:mixture-uniform-functions} applied to
\(\varphi=\mathbf 1_{\{0,1\}}\) gives
\[
    \int_0^1 H_t(\{0,1\})\de t=\lambda(\{0,1\})=0,
\]
we have \(H_t(\{0,1\})=0\) for \(\lambda\)-a.e.~\(t\).
Fix such a \(t\).
Continuity of \(y\mapsto h_y^{(r)}(t)\) follows from dominated convergence in \eqref{eq:main-proof-regularized-h}.
Moreover, if \(0<y_1<y_2<1\), then \(K_r(x,(y_1,y_2])>0\) for every \(x\in(0,1)\), and therefore
\[
    h_{y_2}^{(r)}(t)-h_{y_1}^{(r)}(t)
    =
    \int_0^1 K_r(x,(y_1,y_2])\,K_C(t,\de x)
    >0,
\]
because \(K_C(t,(0,1))=1\).
Hence, for \(\lambda\)-a.e.~\(t\), the map \(y\mapsto h_y^{(r)}(t)\) is continuous on \([0,1]\), strictly increasing on \((0,1)\), and satisfies \(h_0^{(r)}(t)=0\) and \(h_1^{(r)}(t)=1\).
Thus \(C_r\), equipped with the above monotone kernel, is kernel-regular in the sense of Definition~\ref{def:regular}.
Proposition~\ref{prop:regular-proof} therefore gives
\begin{equation}
\label{eq:main-proof-regularized-inequality}
    \xi(C_r)\leq \tau(C_r)
    \qquad
    \text{for every }0<r<1.
\end{equation}

It remains to let \(r\uparrow1\).
To see that \(C_r\to C\), realize the Gaussian transition as follows.
If \((U,V)\sim C\) and \(Z\sim N(0,1)\) is independent of \((U,V)\), set
\[
    W_r:=\Phi(r\Phi^{-1}(V)+\sqrt{1-r^2}\,Z).
\]
Then the copula of \((V,W_r)\) is \(G_r\), and \(W_r\to V\) almost surely as \(r\uparrow1\).
Conditionally on \(V=x\), the variable \(W_r\) has transition kernel
\(K_r(x,\cdot)\). Disintegrating first with respect to \(U\), whose
conditional law of \(V\) is \(K_C(U,\cdot)\), gives precisely the composed
kernel
\[
    t\mapsto \int_0^1 K_r(x,\cdot)\,K_C(t,dx),
\]
which is the Markov kernel of \(C*G_r\).
Hence, if \((U,V)\sim C\), the copula of \((U,W_r)\) is \(C_r=C*G_r\), and \((U,W_r)\to(U,V)\) almost surely.
Consequently, \(C_r\to C\) pointwise, hence uniformly because copulas are uniformly Lipschitz.
Since each \(C_r\) is stochastically increasing and \(C_r\to C\) pointwise, the hypotheses of \cite[Cor.~3.6]{ansari2026continuity} are satisfied.
By the weak continuity of Kendall's tau and by that continuity result for Chatterjee's \(\xi\), it is \(\tau(C_r)\to\tau(C)\) and \(\xi(C_r)\to\xi(C)\).
Letting \(r\uparrow1\) in \eqref{eq:main-proof-regularized-inequality} yields
\(
    \xi(C)\leq\tau(C),
\)
which establishes the first part of the theorem.

Let now \(C^\sigma\) be the copula of \((U,1-V)\), where \((U,V)\sim C\).
If \(C\) is SD, then \(C^\sigma\) is SI.
Indeed, if \(h_y(t)=K_C(t,[0,y])\), define
\(
    h_{a-}(t):=K_C(t,[0,a))=\lim_{z\uparrow a}h_z(t)
\)
for
\(
a\in(0,1].
\)
The reflected kernel may be chosen as
\[
    h_y^\sigma(t)
    =
    K_{C^\sigma}(t,[0,y])
    =
    K_C(t,[1-y,1])
    =
    1-K_C(t,[0,1-y))
    =
    1-h_{(1-y)-}(t).
\]
Since \(C\) is SD, the map \(t\mapsto h_z(t)\) is non-decreasing for every \(z\in(0,1)\).
Hence \(t\mapsto h_{a-}(t)\), being a monotone limit of non-decreasing functions, is non-decreasing for every \(a\in(0,1]\).
Therefore \(t\mapsto h_y^\sigma(t)\) is non-increasing for every \(y\in(0,1)\), and \(C^\sigma\) is SI.
Moreover
\(
    \tau(C^\sigma)=-\tau(C)
\)
and \(\xi(C^\sigma)=\xi(C)\).
The latter follows, for instance, from \(C^\sigma(u,v)=u-C(u,1-v)\), so that \(\partial_1 C^\sigma(u,v)=1-\partial_1 C(u,1-v)\) for a.e.~\((u,v)\), together with \(\int_0^1\int_0^1 \partial_1 C(u,v)\de u\de v=\frac12\).
Applying the SI case just proved to \(C^\sigma\) gives
\[
    \xi(C)=\xi(C^\sigma)\leq \tau(C^\sigma)=-\tau(C).
\]
The equality examples from Proposition~\ref{prop:ordinal-sum} and their reflected versions show sharpness.
\end{proof}

We close this section with two examples that clarify the scope of Theorem~\ref{thm:main}.
Their mass distributions are shown in Figure~\ref{fig:checkerboards}.
\begin{figure}[t!]
  \centering
  \includegraphics[width=\textwidth]{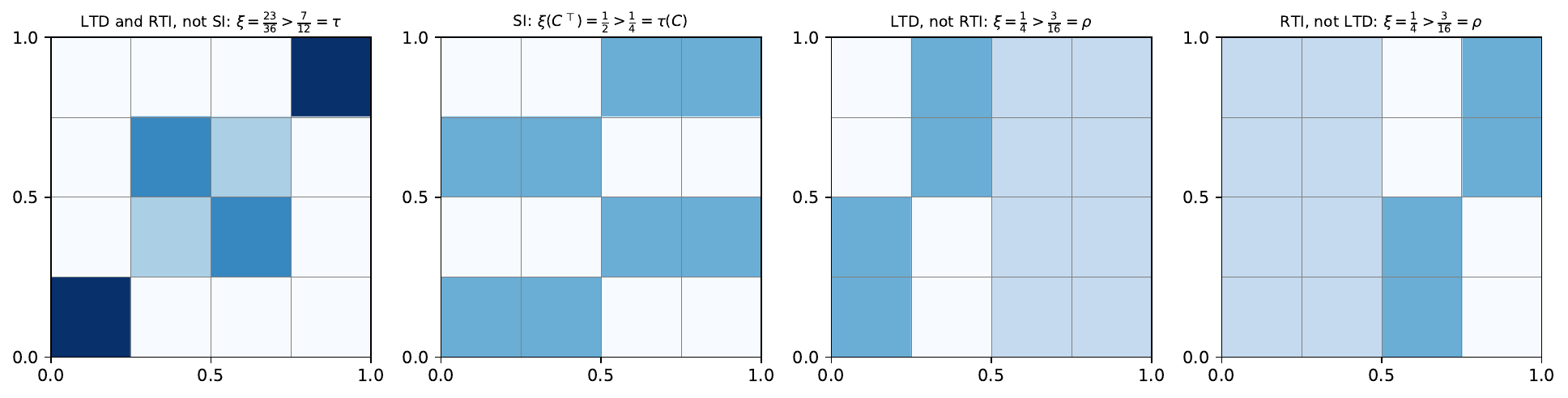}
  \caption{Densities of the four checkerboard copulas used in Example~\ref{rem:ltd-rti-kendall-bound}, Examples~\ref{rem:si-kendall-bound-directional} and~\ref{rem:ltd-rti-spearman-bound-directional} (same copula), Example~\ref{rem:ltd-alone-not-sufficient}, and Example~\ref{rem:rti-alone-not-sufficient} (left to right).
  Darker cells carry more mass, and thin lines mark the cells \(I_i\times I_j\) with \(I_i=[\frac{i-1}{4},\frac{i}{4}]\).
  The first copula is LTD and RTI but not SI and violates \(\xi\leq\tau\); the second is SI, but its reversed coefficient \(\xi(C^\top)\) exceeds both \(\tau(C)\) and \(\rho(C)\).
  The last two satisfy exactly one tail condition each and violate \(\xi\leq\rho\).
  All coefficient values are obtained in closed form from the checkerboard formulas of \cite[Prop.~3.3]{rockel2025measures}.}
  \label{fig:checkerboards}
\end{figure}
First, the SI assumption cannot be weakened to the LTD and RTI tail conditions appearing in Theorem~\ref{thm:ltd-rti-xi-rho}.
Second, the Kendall bound is directional: even when \(\xi(C)\leq\tau(C)\), one need not have \(\xi(C^\top)\leq\tau(C)\).

In all checkerboard examples below, the displayed intervals
\(
    I_i=\left[\frac{i-1}{4},\frac{i}{4}\right]
\)
are understood with the half-open convention
\(
    I_i=\left[\frac{i-1}{4},\frac{i}{4}\right)
\)
for \(i=1,2,3\) and
\(
    I_4=\left[\frac34,1\right].
\)
Values of the conditional kernel on the finitely many grid boundaries may be chosen arbitrarily, or by the right-continuous strip convention.
These boundary choices are \(\lambda\)-null and do not affect the copula or any of the rank coefficients.
All stripwise monotonicity checks below refer to this chosen version.

\begin{example}[LTD and RTI jointly do not suffice for $\xi\leq\tau$]
\label{rem:ltd-rti-kendall-bound}
Theorem~\ref{thm:ltd-rti-xi-rho} shows that the same-direction LTD and RTI assumptions are sufficient for the Spearman bound \(\xi(C)\leq\rho(C)\).
They are not, however, sufficient for the sharper Kendall bound \(\xi(C)\leq\tau(C)\).
Let \(C\) be the checkerboard copula on the partition
\(
    I_i=\left[\frac{i-1}{4},\frac{i}{4}\right]
\)
for
\(
    i=1,\dots,4,
\)
with cell probabilities \(p_{ij}=m_{ij}/4\), where
\[
    M=(m_{ij})_{i,j=1}^4
    =
    \begin{pmatrix}
        1&0&0&0\\
        0&1/3&2/3&0\\
        0&2/3&1/3&0\\
        0&0&0&1
    \end{pmatrix}.
\]
Equivalently, \(C\) has density \(4m_{ij}\) on \(I_i\times I_j\), shown in the first panel of Figure~\ref{fig:checkerboards}.
Since \(M\) is doubly stochastic, \(C\) is a copula.

We first verify the LTD and RTI conditions.
A version of the conditional kernel is constant on each vertical strip.
For \(u\in I_i\), write
\(
    H_i(y)=K_C(u,[0,y]),
\)
and for fixed \(y\), set \(a_i(y)=H_i(y)\).
Then
\[
\begin{array}{c|c}
 y\text{-range} & (a_1(y),a_2(y),a_3(y),a_4(y))\\
\hline
 0\leq y\leq 1/4
    & (r,0,0,0),\quad r=4y,\\[2mm]
 1/4\leq y\leq 1/2
    & \left(1,\frac r3,\frac{2r}{3},0\right),\quad r=4y-1,\\[2mm]
 1/2\leq y\leq 3/4
    & \left(1,\frac{1+2r}{3},\frac{2+r}{3},0\right),\quad r=4y-2,\\[2mm]
 3/4\leq y\leq 1
    & (1,1,1,r),\quad r=4y-3.
\end{array}
\]
For every \(y\), these vectors satisfy
\[
    a_m(y)\leq \frac{1}{m-1}\sum_{i<m}a_i(y),
    \qquad m=2,3,4,
\]
and
\[
    a_m(y)\geq \frac{1}{4-m}\sum_{i>m}a_i(y),
    \qquad m=1,2,3.
\]
Indeed, these inequalities are checked directly from the four displayed cases.
For a step function with equal-width steps \(a_1(y),\dots,a_4(y)\), the first set of inequalities is precisely the condition that the left-tail averages are non-increasing, and the second set is precisely the condition that the right-tail averages are non-increasing.
Hence, for every \(y\),
\[
    u\longmapsto \Prob(V\leq y\mid U\leq u)
    \quad\text{and}\quad
    u\longmapsto \Prob(V\leq y\mid U>u)
\]
are non-increasing.
Thus \(C\) is LTD and RTI in the stated direction.
Since \(M\) is symmetric, the same is true with \(U\) and \(V\) interchanged.
However, \(C\) is not SI.
Indeed,
\(
    H_2(1/2)=\frac13
\)
whereas
\(
    H_3(1/2)=\frac23,
\)
so the map \(u\mapsto K_C(u,[0,1/2])\) is not non-increasing.

It remains to compare the two coefficients.
Let
\(
    \Delta=(p_{ij})_{i,j=1}^4=\frac14 M
\)
be the associated \(4\times4\) checkerboard matrix.
Applying the explicit checkerboard formulas from \cite[Prop.~3.3]{rockel2025measures} with \(m=n=4\) gives
\(
    \xi(C)
    =
    \frac{23}{36},
\)
and
\(
    \tau(C)
    =
    \frac{7}{12}.
\)
Consequently,
\(
    \xi(C)-\tau(C)
    =
    \frac1{18}>0.
\)
Thus, the LTD and RTI tail conditions are sufficient for the Spearman bounds \(\tau\leq\rho\) and \(\xi\leq\rho\), but not for the sharper Kendall bound \(\xi\leq\tau\).
\end{example}

\begin{example}[The SI Kendall bound is directional]
\label{rem:si-kendall-bound-directional}
The implication in Theorem~\ref{thm:main} is genuinely directional.
In general, stochastic increasingness of \(C\) does not imply that
\(
    \xi(C^\top)\leq \tau(C).
\)
Let \(C\) be the checkerboard copula on the partition
\(
    I_i=\left[\frac{i-1}{4},\frac{i}{4}\right]
\)
for
\(
 i=1,\dots,4,
\)
with cell probabilities \(p_{ij}=m_{ij}/4\), where
\[
    M=(m_{ij})_{i,j=1}^4
    =
    \begin{pmatrix}
        1/2&0&1/2&0\\
        1/2&0&1/2&0\\
        0&1/2&0&1/2\\
        0&1/2&0&1/2
    \end{pmatrix}.
\]
Since \(M\) is doubly stochastic, this defines a copula, and its density is shown in the second panel of Figure~\ref{fig:checkerboards}.
We first check that \(C\) is SI.
For \(u\in I_1\cup I_2\), the conditional law is
\(
    A:=\frac12\,\mathrm{Unif}(I_1)+\frac12\,\mathrm{Unif}(I_3),
\)
whereas for \(u\in I_3\cup I_4\), it is
\(
    B:=\frac12\,\mathrm{Unif}(I_2)+\frac12\,\mathrm{Unif}(I_4).
\)
The distribution functions satisfy
\(
    F_A(y)\geq F_B(y)
\)
for
\(
0\leq y\leq1.
\)
Thus \(u\mapsto K_C(u,[0,y])\) is non-increasing for every \(y\), and hence
\(C\) is SI.
However, for the transpose direction, applying the explicit checkerboard formulas from \cite[Prop.~3.3]{rockel2025measures} with
\(
    \Delta=(p_{ij})_{i,j=1}^4=\frac14 M
\)
gives
\(
    \xi(C^\top)=\frac12
\)
and
\(
    \tau(C)=\frac14.
\)
Consequently,
\(
    \xi(C^\top)>\tau(C).
\)
Thus, the Kendall bound in Theorem~\ref{thm:main} holds for the same direction in which the copula is stochastically increasing, but not necessarily for the reversed Chatterjee coefficient.
\end{example}

\section{The Spearman bound under LTD and RTI}
\label{sec:ltd-rti-rho}

In this section, we prove Theorem~\ref{thm:ltd-rti-xi-rho}.
For the equality case, we first establish two auxiliary one-dimensional lemmas that characterize the extremal behavior in the estimates underlying Theorem~\ref{thm:ltd-rti-xi-rho}.
For \(g\in L^\infty(0,1)\) with antiderivative \(G(u):=\int_0^u g(t)\de t\), define
\begin{equation}\label{eq:one-dimensional-quantities}
    Q(g):=\int_0^1 g(u)^2\de u,
    \qquad
    B(g):=\int_0^1 G(u)\de u,
    \qquad
    P(g):=\int_0^1 g(u)_+\de u,
\end{equation}
where \(g_+(u):=\max\{g(u),0\}\) and similarly \(g_-(u):=-\min\{g(u),0\}\).

\begin{lemma}
\label{lem:xi-rho-section-estimate}
    Fix \(v\in(0,1)\), and let \(g\in L^\infty(0,1)\) satisfy \(-v\leq g(u)\leq 1-v\) for a.e.~\(u\in(0,1)\), \(\int_0^1 g(u)\de u=0\), and
\[
    G(u)\geq0,
    \qquad
    ug(u)\leq G(u),
    \qquad
    -(1-u)g(u)\leq G(u)
    \qquad\text{for a.e.~}u\in(0,1).
\]
    Then
    \begin{align}
    \label{eq:one-dimensional-QP-deficit}
        P(g)-Q(g)
        &=
        \int_0^1 g_+(u)\bigl((1-v)-g_+(u)\bigr)\de u
        +
        \int_0^1 g_-(u)\bigl(v-g_-(u)\bigr)\de u, \\
    \label{eq:one-dimensional-PB-deficit}
        2B(g)-P(g)
        &=
        \int_{\{g>0\}}G(u)-ug(u)\de u
        +
        \int_{\{g<0\}}G(u)+(1-u)g(u)\de u
        +
        \int_{\{g=0\}}G(u)\de u,
    \end{align}
    and in particular \(Q(g)\leq P(g)\leq 2B(g)\).
\end{lemma}

\begin{proof}
    Since \(\int_0^1 g(u)\de u=0\), we have
    \(
        P(g)
        =
        \int_0^1 g_+(u)\de u
        =
        \int_0^1 g_-(u)\de u.
    \)
    Hence
    \[
    \begin{aligned}
        P(g)-Q(g)
        &=
        (1-v)\int_0^1 g_+(u)\de u
        +
        v\int_0^1 g_-(u)\de u
        -
        \int_0^1\bigl(g_+(u)^2+g_-(u)^2\bigr)\de u  \\
        &=
        \int_0^1 g_+(u)\bigl((1-v)-g_+(u)\bigr)\de u
        +
        \int_0^1 g_-(u)\bigl(v-g_-(u)\bigr)\de u.
    \end{aligned}
    \]
    This proves \eqref{eq:one-dimensional-QP-deficit}, and the right-hand
    side is nonnegative because \(g_+(u)\leq1-v\) and \(g_-(u)\leq v\) for
    a.e.~\(u\).
    Moreover, \(G(0)=G(1)=0\), and integration by parts gives
    \[
        B(g)
        =
        \int_0^1 G(u)\de u
        =
        \int_0^1 (1-u)g(u)\de u.
    \]
    Therefore
    \[
    \begin{aligned}
        P(g)-B(g)
        &=
        \int_{\{g>0\}} u g(u)\de u
        +
        \int_{\{g<0\}} (1-u)(-g(u))\de u.
    \end{aligned}
    \]
    Subtracting this identity from \(B(g)=\int_0^1G(u)\de u\) gives
    \eqref{eq:one-dimensional-PB-deficit}. Its right-hand side is
    nonnegative by the assumptions \(G(u)\geq0\), \(ug(u)\leq G(u)\), and
    \(-(1-u)g(u)\leq G(u)\), each holding for a.e.~\(u\). Thus
    \(P(g)\leq2B(g)\), and the proof is complete.
\end{proof}

\begin{lemma}
\label{lem:xi-rho-equality-section}
Fix \(v\in(0,1)\), and let \(g\in L^\infty(0,1)\) satisfy \(-v\leq g(u)\leq 1-v\) for a.e.~\(u\in(0,1)\), \(\int_0^1 g(u)\de u=0\), and
\[
    G(u)\geq0,
    \qquad
    ug(u)\leq G(u),
    \qquad
    -(1-u)g(u)\leq G(u)
    \qquad\text{for a.e.~}u\in(0,1).
\]
If \(Q(g) = 2B(g)\), then either
\[
    g(u)=0
    \quad\text{ for a.e.~}u\in(0,1),
    \qquad\text{or}\qquad
    g(u)
    =
    \begin{cases}
        1-v, & 0<u<v,\\
        -v, & v<u<1,
    \end{cases}
    \quad\text{for a.e.~}u\in(0,1).
\]
\end{lemma}

Lemma~\ref{lem:xi-rho-equality-section} will be applied for the equality case sectionwise in the second coordinate, with
\[
g(u)=\partial_1 C(u,v)-v,
\qquad
G(u)=C(u,v)-uv.
\]
The LTD and RTI assumptions translate exactly into the one-dimensional sign and monotonicity constraints imposed on \(g\) and \(G\).

\begin{proof}
By Lemma~\ref{lem:xi-rho-section-estimate},
\(
    Q(g)\leq P(g)\leq2B(g).
\)
Since \(Q(g)=2B(g)\), equality holds in both inequalities:
\(
    Q(g)=P(g)=2B(g).
\)
The deficit identity \eqref{eq:one-dimensional-QP-deficit} therefore gives
\[
    g_+(u)\bigl((1-v)-g_+(u)\bigr)=0
    \qquad\text{and}\qquad
    g_-(u)\bigl(v-g_-(u)\bigr)=0
\]
for a.e.~\(u\). Hence
\begin{equation}
\label{eq:equality-section-three-values}
    g(u)\in\{-v,0,1-v\}
    \qquad\text{for a.e.~}u.
\end{equation}

Let
\[
    A_+:=\{u:g(u)=1-v\},
    \qquad
    A_0:=\{u:g(u)=0\},
    \qquad
    A_-:=\{u:g(u)=-v\},
\]
where the sets are understood up to null sets. Then
\[
    P(g)=(1-v)\lambda(A_+)=v\lambda(A_-).
\]

Since \(P(g)=2B(g)\), the deficit identity
\eqref{eq:one-dimensional-PB-deficit} also vanishes. Its three
nonnegative terms must therefore vanish separately. Using
\eqref{eq:equality-section-three-values}, we obtain
\begin{align}
\label{eq:equality-section-G-values}
    G(u)&=u(1-v)
    &&\text{for a.e.~}u\in A_+,\notag\\
    G(u)&=(1-u)v
    &&\text{for a.e.~}u\in A_-,\\
    G(u)&=0
    &&\text{for a.e.~}u\in A_0.\notag
\end{align}
We now show that \(A_+\) is an initial interval and \(A_-\) is a terminal interval, up to null sets.
Let
\[
    E_+:=A_+\cap\{u:G(u)=u(1-v)\}.
\]
Then \(A_+\setminus E_+\) is null, hence \(\operatorname*{ess\,sup}E_+=\operatorname*{ess\,sup}A_+=s\).
For \(u\in E_+\),
\[
    \int_0^u \bigl((1-v)-g(t)\bigr)\de t=0.
\]
Since \((1-v)-g(t)\ge0\) a.e., it follows that \(g(t)=1-v\) for a.e.~\(t\in(0,u)\).
Thus, for every rational \(r<s\), one can choose \(u\in E_+\) with \(u>r\), and hence \(g=1-v\) a.e.~on \((0,r)\). Letting \(r\uparrow s\) along rationals gives \(A_+=(0,s)\) up to a null set.
Thus \(A_+\) is an initial interval up to null sets.
Similarly, let \(E_-:=\{u\in A_-:G(u)=(1-u)v\}\).
For every \(u\in E_-\), using \(G(1)=0\) and \eqref{eq:equality-section-G-values}, we have
\[
    -\int_u^1 g(t)\de t
    =
    G(u)
    =
    (1-u)v.
\]
Since \(-g\leq v\) a.e., this is possible only if
\(
    g(t)=-v
    \text{ for a.e.~}t\in(u,1).
\)
Let
\(
    t:=\operatorname*{ess\,inf} A_-.
\)
Then
\(
    A_-=(t,1)
    \text{ up to a null set}.
\)
Since the three level sets \(A_+\), \(A_0\), and \(A_-\) are disjoint and cover \((0,1)\) up to a null set, we have \(0\leq s\leq t\leq1\), and \(A_0=(s,t)\) up to a null set.
On these intervals,
\begin{equation}
\label{eq:equality-section-step-form}
    g(u)
    =
    \begin{cases}
        1-v, & 0<u<s,\\
        0, & s<u<t,\\
        -v, & t<u<1,
    \end{cases}
    \qquad\text{a.e.}
\end{equation}
The condition \(\int_0^1 g=0\) gives
\(
    (1-v)s=v(1-t).
\)
Moreover, by the definition of \(P\) and the step form
\eqref{eq:equality-section-step-form},
\[
    P(g)=(1-v)s=v(1-t).
\]
If \(P(g)=0\), then \(s=0\) and \(t=1\), and therefore
\eqref{eq:equality-section-step-form} gives \(g=0\) a.e.
This gives the first alternative.
Assume now \(P(g)>0\).
By \eqref{eq:equality-section-G-values}, \(G=0\) a.e.~on \(A_0=(s,t)\), while
\eqref{eq:equality-section-step-form} gives
\[
    G(u)=P(g)
    \qquad\text{for }s<u<t.
\]
Hence \(s=t\).
Hence
\(
    (1-v)s=v(1-s),
\)
and \(s=v\).
Consequently, \eqref{eq:equality-section-step-form} gives
\[
    g(u)
    =
    \begin{cases}
        1-v, & 0<u<v,\\
        -v, & v<u<1,
    \end{cases}
    \qquad\text{for a.e.~}u\in(0,1).
\]
This is the second alternative, and the proof is complete.
\end{proof}

\begin{proof}[Proof of Theorem~\ref{thm:ltd-rti-xi-rho}]
Fix \(v\in(0,1)\).
We use the kernel version \(h_v(u)=K_C(u,[0,v])\) from \eqref{eq:h-definitions}.
For fixed \(v\), the identity
\[
    C(u,v)=\int_0^u h_v(t)\de t
\]
implies \(h_v=\partial_1 C(\cdot,v)\) for a.e.~\(u\). All subsequent
differential inequalities are hence understood a.e.~in \(u\).
Put \(g_v(u):=h_v(u)-v\) and \(G_v(u):=C(u,v)-uv=\int_0^u g_v(t)\de t\).
Since \(C(1,v)=v\), we have
\(
    \int_0^1 g_v(u)\de u=0.
\)
The LTD assumption \eqref{eq:ltd_definition} says that \(L_v(u):=C(u,v)/u=v+G_v(u)/u\) is non-increasing on \((0,1]\).
Since \(L_v(1)=v\), this implies
\begin{equation}
\label{eq:ltd-rti-Gv-positive}
    G_v(u)\geq 0,
    \qquad 0\leq u\leq 1.
\end{equation}
Moreover, differentiating \(L_v\) at points where \(h_v\) exists gives
\begin{equation}
\label{eq:ltd-rti-left-differential-bound}
    u g_v(u)\leq G_v(u)
    \qquad\text{for a.e.~}u\in(0,1).
\end{equation}
Similarly, the RTI assumption says that \(R_v(u):=(v-C(u,v))/(1-u)=v-G_v(u)/(1-u)\) is non-increasing on \([0,1)\).
Differentiating \(R_v\) gives
\begin{equation}
\label{eq:ltd-rti-right-differential-bound}
    -(1-u)g_v(u)\leq G_v(u)
    \qquad\text{for a.e.~}u\in(0,1).
\end{equation}
Recall \eqref{eq:one-dimensional-quantities} and write
\[
    Q_v:=Q(g_v)=\int_0^1 g_v(u)^2\de u,
    \quad
    B_v:=B(g_v)=\int_0^1 G_v(u)\de u,
    \quad
    P_v:=P(g_v)=\int_0^1 (g_v(u))_+\de u.
\]
Since \(0\leq h_v\leq1\), we have
\[
    -v\leq g_v(u)\leq1-v
    \qquad\text{for a.e.~}u.
\]
Together with \(\int_0^1 g_v\de u=0\),
\eqref{eq:ltd-rti-Gv-positive},
\eqref{eq:ltd-rti-left-differential-bound}, and
\eqref{eq:ltd-rti-right-differential-bound}, this shows that \(g_v\)
satisfies the hypotheses of Lemma~\ref{lem:xi-rho-section-estimate}.
Therefore
\[
    0\leq Q_v\leq P_v\leq2B_v.
\]
In particular,
\begin{equation}
\label{eq:ltd-rti-section-estimate}
    \int_0^1 (h_v(u)-v)^2\de u
    =
    Q_v
    \leq
    2B_v
    =
    2\int_0^1\bigl(C(u,v)-uv\bigr)\de u.
\end{equation}
Integrating \eqref{eq:ltd-rti-section-estimate} over \(v\) gives the result.
Indeed, using \(\int_0^1 h_v(u)\de u=v\), we have
\[
    \xi(C)
    =
    6\int_0^1\int_0^1 h_v(u)^2\de u\de v-2
    =
    6\int_0^1\int_0^1 (h_v(u)-v)^2\de u\de v,
\]
whereas \eqref{eq:rho-definition} gives \(\rho(C)=12\int_0^1\int_0^1 (C(u,v)-uv)\de u\de v\).
Therefore
\(
    \xi(C)\leq\rho(C).
\)

Assume now that equality \(\xi(C)=\rho(C)\) holds. Since
\[
    \xi(C)
    =
    6\int_0^1 Q_v\de v,
    \qquad
    \rho(C)
    =
    12\int_0^1 B_v\de v,
\]
we have
\(
    \int_0^1 \bigl(2B_v-Q_v\bigr)\de v=0.
\)
By Lemma~\ref{lem:xi-rho-section-estimate}, \(2B_v-Q_v\geq0\) for every
\(v\in(0,1)\).
Hence,
\(
    Q_v=2B_v
    \text{ for a.e.~}v\in(0,1).
\)
Applying Lemma~\ref{lem:xi-rho-equality-section} sectionwise, for almost
every \(v\in(0,1)\) one has either
\[
    h_v(u)=v
    \quad\text{for a.e.~}u,
    \qquad\text{or}\qquad
    h_v(u)=\mathbf 1_{\{u\leq v\}}
    \quad\text{for a.e.~}u.
\]
Let
\(
    V_\Pi
    :=
    \{
        v\in(0,1):
        \int_0^1 |h_v(u)-v|\de u=0
    \}
\)
and
\(
    V_M
    :=
    \{
        v\in(0,1):
        \int_0^1
            \left|h_v(u)-\mathbf 1_{\{u\leq v\}}\right|
        \de u=0
    \}.
\)
The preceding paragraph says that \(V_\Pi\cup V_M\) has full Lebesgue measure in \((0,1)\).

We claim that \(V_\Pi\) and \(V_M\) cannot both have positive measure.
Indeed, if both had positive measure, we could choose \(a\in V_\Pi\) and \(b\in V_M\) with \(a\neq b\).
Let \(N_a,N_b\subseteq[0,1]\) be \(\lambda\)-null sets such that
\[
    h_a(u)=a
    \quad\text{for }u\notin N_a,
    \qquad
    h_b(u)=\mathbf 1_{\{u\leq b\}}
    \quad\text{for }u\notin N_b.
\]
Since \(v\mapsto h_v(u)=K_C(u,[0,v])\) is non-decreasing for every \(u\), the following alternatives are impossible.
If \(b<a\), then for every
\(
    u\in (0,b)\setminus(N_a\cup N_b)
\)
we have
\[
    h_b(u)=1>a=h_a(u),
\]
contradicting \(h_b(u)\leq h_a(u)\).
If \(a<b\), then for every
\(
    u\in (b,1)\setminus(N_a\cup N_b)
\)
we have
\[
    h_a(u)=a>0=h_b(u),
\]
contradicting \(h_a(u)\leq h_b(u)\).
Therefore either \(V_\Pi\) or \(V_M\) is null.
If \(V_M\) is null, then \(h_v(u)=v\) for \(\lambda^2\)-a.e.~\((u,v)\).
By Fubini's theorem, there is a full-measure set \(E\subseteq(0,1)\) such
that, for every \(v\in E\), the identity \(h_v(u)=v\) holds for
\(\lambda\)-a.e.~\(u\). Hence, for every \(v\in E\) and every
\(u\in[0,1]\),
\[
    C(u,v)
    =
    \int_0^u h_v(t)\de t
    =
    uv.
\]
Since \(E\) has full measure in \((0,1)\), it is dense in \([0,1]\). For
fixed \(u\), both maps \(v\mapsto C(u,v)\) and \(v\mapsto uv\) are continuous,
so the identity extends to every \(v\in[0,1]\). Thus \(C=\Pi\).

If \(V_\Pi\) is null, then
\(h_v(u)=\mathbf 1_{\{u\leq v\}}\) for \(\lambda^2\)-a.e.~\((u,v)\).
By Fubini's theorem, there is a full-measure set \(E\subseteq(0,1)\) such
that, for every \(v\in E\), this identity holds for \(\lambda\)-a.e.~\(u\).
Hence, for every \(v\in E\) and every \(u\in[0,1]\),
\[
    C(u,v)
    =
    \int_0^u \mathbf 1_{\{t\leq v\}}\de t
    =
    \min\{u,v\}.
\]
Again \(E\) is dense, and continuity in \(v\) extends the identity to every
\(v\in[0,1]\). Thus \(C=M\).

It remains to prove the statement for LTI and RTD copulas.
Assume that \(C\) is LTI and RTD.
Let \(C^\sigma\) be the copula of \((U,1-V)\), where \((U,V)\sim C\), that is, \(C^\sigma(u,v)=u-C(u,1-v)\).
If \(C\) is LTI and RTD, then \(C^\sigma\) is LTD and RTI: both \(C^\sigma(u,v)/u=1-C(u,1-v)/u\) and \((v-C^\sigma(u,v))/(1-u)=1-((1-v)-C(u,1-v))/(1-u)\) are non-increasing in \(u\).
Applying the first part to \(C^\sigma\) gives
\(
    \xi(C^\sigma)\leq \rho(C^\sigma).
\)
Moreover, \(\xi(C^\sigma)=\xi(C)\) and \(\rho(C^\sigma)=-\rho(C)\).
Hence
\(
    \xi(C)\leq -\rho(C).
\)
The equality statement follows from the equality case in the first part, since
\(C^\sigma=\Pi\) is equivalent to \(C=\Pi\), while \(C^\sigma=M\) is equivalent
to \(C=W\).

Conversely, \(\Pi\) and \(M\) are LTD and RTI and satisfy \(\xi(\Pi)=\rho(\Pi)=0\), \(\xi(M)=\rho(M)=1\).
Similarly, \(\Pi\) and \(W\) are LTI and RTD and satisfy \(\xi(\Pi)=-\rho(\Pi)=0\), \(\xi(W)=-\rho(W)=1\).
\end{proof}

We conclude the paper with examples illustrating that both tail conditions in Theorem~\ref{thm:ltd-rti-xi-rho} are needed and that the resulting Spearman bound is directional.

\begin{example}[LTD alone does not imply \(\xi\leq\rho\)]
\label{rem:ltd-alone-not-sufficient}
The RTI assumption in Theorem~\ref{thm:ltd-rti-xi-rho} cannot be omitted.
Let \(C\) be the checkerboard copula on the partition
\(
    I_i=\left[\frac{i-1}{4},\frac{i}{4}\right]
\)
for
\(
    i=1,\dots,4,
\)
with cell probabilities \(p_{ij}=m_{ij}/4\), where
\[
    M=(m_{ij})_{i,j=1}^4
    =
    \begin{pmatrix}
        1/2&1/2&0&0\\
        0&0&1/2&1/2\\
        1/4&1/4&1/4&1/4\\
        1/4&1/4&1/4&1/4
    \end{pmatrix}.
\]
Since \(M\) is doubly stochastic, this defines a copula.
Its density is shown in the third panel of Figure~\ref{fig:checkerboards}.
For \(u\in I_i\), write
\(
    H_i(y)=K_C(u,[0,y]).
\)
Then the four conditional distribution functions have the form
\[
    (H_1(y),H_2(y),H_3(y),H_4(y))
    =
    (A(y),B(y),y,y),
\]
where
\(
    A(y)=\min\{2y,1\}
\)
and
\(
    B(y)=\max\{2y-1,0\}.
\)
In particular, one has
\(
    B(y)\leq A(y)
\)
and
\(
    A(y)+B(y)=2y.
\)
Hence, for every \(y\),
\[
    H_2(y)\leq H_1(y),
    \qquad
    H_3(y)=\frac{H_1(y)+H_2(y)}{2},
    \qquad
    H_4(y)=\frac{H_1(y)+H_2(y)+H_3(y)}{3}.
\]
Thus, the left-tail averages are non-increasing, and \(C\) is LTD.
However, \(C\) is not RTI.
Indeed, at \(y=1/2\) one has
\[
    (H_1(1/2),H_2(1/2),H_3(1/2),H_4(1/2))
    =
    \left(1,0,\frac12,\frac12\right),
\]
so the second strip value is smaller than the average of the following right-tail values.
Consider the associated \(4\times4\) checkerboard matrix
\(
    \Delta=(p_{ij})_{i,j=1}^4=\frac14 M.
\)
Applying the explicit checkerboard formulas from \cite[Prop.~3.3]{rockel2025measures} gives
\(
    \xi(C)=\frac14
\)
and
\(
    \rho(C)=\frac{3}{16}.
\)
Consequently,
\(
    \xi(C)-\rho(C)
    =
    \frac1{16}>0.
\)
Thus LTD alone is not sufficient for \(\xi(C)\leq\rho(C)\).
\end{example}

\begin{example}[RTI alone does not imply \(\xi\leq\rho\)]
\label{rem:rti-alone-not-sufficient}
The LTD assumption in Theorem~\ref{thm:ltd-rti-xi-rho} cannot be omitted.
Let \(C\) be the checkerboard copula with matrix
\[
    M
    =
    \begin{pmatrix}
        1/4&1/4&1/4&1/4\\
        1/4&1/4&1/4&1/4\\
        1/2&1/2&0&0\\
        0&0&1/2&1/2
    \end{pmatrix}.
\]
This is the \(180^\circ\)-rotation of the checkerboard matrix in
Example~\ref{rem:ltd-alone-not-sufficient}, and is again doubly stochastic.
Hence it defines a copula, and its density is shown in the fourth panel of Figure~\ref{fig:checkerboards}.

We verify that \(C\) is RTI.
As before, a version of the conditional kernel is constant on each vertical strip.
For \(u\in I_i\), write
\[
    H_i(y)=K_C(u,[0,y]).
\]
For fixed \(y\), set \(a_i(y)=H_i(y)\).
Since the first two rows of \(M\) are uniform, while the third and fourth rows are uniform on the lower and upper halves, respectively, we have
\[
    a_1(y)=y,
    \quad
    a_2(y)=y,
    \quad
    a_3(y)=
    \begin{cases}
        2y, & 0\leq y\leq 1/2,\\
        1, & 1/2\leq y\leq1,
    \end{cases}
    \quad
    a_4(y)=
    \begin{cases}
        0, & 0\leq y\leq1/2,\\
        2y-1, & 1/2\leq y\leq1.
    \end{cases}
\]
For a step function with equal-width steps, RTI is equivalent to the condition that each step is at least the average of the steps to its right, namely
\[
    a_m(y)
    \geq
    \frac{1}{4-m}\sum_{i>m}a_i(y),
    \qquad m=1,2,3.
\]
These inequalities are immediate from the preceding formulas.
Indeed, if
\(0\leq y\leq1/2\), then
\[
    a_1(y)=y
    =
    \frac{y+2y+0}{3},
    \qquad
    a_2(y)=y
    =
    \frac{2y+0}{2},
    \qquad
    a_3(y)=2y\geq0=a_4(y).
\]
If \(1/2\leq y\leq1\), then
\[
    a_1(y)=y
    =
    \frac{y+1+(2y-1)}{3},
    \quad
    a_2(y)=y
    =
    \frac{1+(2y-1)}{2},
    \quad
    a_3(y)=1\geq2y-1=a_4(y).
\]
Consequently, for every \(y\),
\(
    u\longmapsto \Prob(V\leq y\mid U>u)
\)
is non-increasing, and hence \(C\) is RTI.
However, \(C\) is not LTD.
Taking \(y=1/2\), the left-tail averages
satisfy
\[
    \Prob(V\leq1/2\mid U\leq1/2)
    =
    \frac{a_1(1/2)+a_2(1/2)}{2}
    =
    \frac12,
\]
whereas
\[
    \Prob(V\leq1/2\mid U\leq3/4)
    =
    \frac{a_1(1/2)+a_2(1/2)+a_3(1/2)}{3}
    =
    \frac23.
\]
Thus the map
\(
    u\longmapsto \frac{C(u,1/2)}{u}
    =
    \Prob(V\leq1/2\mid U\leq u)
\)
is not non-increasing, so \(C\) is not LTD.
Finally, applying the checkerboard formulas gives
\(
    \xi(C)=\frac14
\)
and
\(
    \rho(C)=\frac{3}{16}.
\)
Therefore
\(
    \xi(C)-\rho(C)=\frac1{16}>0,
\)
so RTI alone is not sufficient for \(\xi(C)\leq\rho(C)\).
\end{example}

\begin{example}[The Spearman bound is directional]
\label{rem:ltd-rti-spearman-bound-directional}
The implication in Theorem~\ref{thm:ltd-rti-xi-rho} is also genuinely
directional.
In general, LTD and RTI do \emph{not} imply
\(
    \xi(C^\top)\leq\rho(C).
\)

Let \(C\) be the checkerboard copula on the partition
\(
    I_i=\left[\frac{i-1}{4},\frac{i}{4}\right]
\)
for
\(
    i=1,\dots,4,
\)
with cell probabilities \(p_{ij}=m_{ij}/4\), where
\[
    M=(m_{ij})_{i,j=1}^4
    =
    \begin{pmatrix}
        1/2&0&1/2&0\\
        1/2&0&1/2&0\\
        0&1/2&0&1/2\\
        0&1/2&0&1/2
    \end{pmatrix}.
\]
\(M\) is doubly stochastic and \(C\) is SI, so \(C\) is in particular LTD and RTI.
$C$ is the same as the copula of Example~\ref{rem:si-kendall-bound-directional}, shown in the second panel of Figure~\ref{fig:checkerboards}.
Applying the explicit checkerboard formulas from \cite[Prop.~3.3]{rockel2025measures} with
\[
    \Delta=(p_{ij})_{i,j=1}^4=\frac14 M
\]
gives
\(
    \xi(C^\top)=\frac12
\)
and
\(
    \rho(C)=\frac38.
\)
Consequently,
\(
    \xi(C^\top)>\rho(C).
\)
Thus, the Spearman bound in Theorem~\ref{thm:ltd-rti-xi-rho} holds for \(\xi(C)\), but not necessarily for the reversed coefficient \(\xi(C^\top)\).
\end{example}

\bibliographystyle{plainnat}
\bibliography{Literature}

\end{document}